\documentclass[english, 11pt]{article}
\usepackage[T1]{fontenc}
\usepackage[latin1]{inputenc}
\usepackage{amssymb, amsmath, amsthm, latexsym, setspace}
\usepackage{hyperref,graphicx}
\usepackage[francais,english]{babel}
\usepackage[margin=1in]{geometry}
\newtheorem{thm}{Theorem}
\newtheorem{lem}{Lemma}

\newtheorem{defi}{Definition}

\newcommand{\Prob}{\mathbb{P}}
\newcommand{\E}{\mathbb{E}}

\newcommand{\eps}{\varepsilon}
\newcommand{\ZZ}{\mathcal{X}}
\newcommand{\lf}{\lfloor}
\newcommand{\rf}{\rfloor}
\newcommand{\set}[1]{\left\{#1\right\}}
\newcommand{\Zdies}{\hbox{ dies out}}
\newcommand{\Zdiest}{\hbox{ d.o.}}
\newcommand{\N}{\mathbb{N}}

\author{Olivier Couronn\'e , Lucas Gerin\\ \scriptsize Universit\'e Paris Ouest Nanterre La D\'efense. Equipe Modal'X.\\ \scriptsize 200 avenue de la R\'epublique. 92000 Nanterre (France).}
\title{A Branching-selection process related to censored Galton-Walton processes}

\begin{document}

\bibliographystyle{plain}
\maketitle
\selectlanguage{english}
\begin{abstract}
We obtain the asymptotics for the speed of a particular case of a particle system with branching and selection introduced by B\'erard and Gou\'er\'e (2010). 
The proof is based on a connection with a supercritical Galton-Watson process \emph{censored} at a certain level.
\end{abstract}
\selectlanguage{francais}
\begin{abstract}
Nous \'etudions un cas particulier de syst\`eme de particules avec branchement et s\'election introduit par B\'erard et Gou\'er\'e (2010). Nous obtenons l'asympotique pour la vitesse, 
en remarquant un lien avec un processus de Galton-Watson surcritique censur\'e \`a un certain niveau.
\end{abstract}
\selectlanguage{english}


\section{Models and Results}
\subsection{The censored Galton-Watson process}
For a probability distribution $\ZZ$ on non-negative integers, the Galton-Watson process with \emph{offspring} $\ZZ$ is the process
defined by $Z_0^a=a$ and
$$
Z^a_{k+1}=\sum_{i=1}^{Z^a_{k}} X_{k,i}
$$
where the $X_{k,i}$'s are i.i.d. copies of $\ZZ$ ($Z_{k+1}^a=0$ if $Z_k^a=0$). We write $(Z_k)=(Z^1_k)$.
We deal here with supercritical Galton-Watson processes (\emph{i.e.} when $\mathbb{E}(\ZZ)>1$) and to avoid trivialities we impose $\ZZ(0)>0$. In this case, it is well-known (we refer to \cite{MR2047480} for basic facts on branching processes) that $Z_k$ dies with a certain probability $0<q<1$ which is the unique solution in $(0,1)$ of
$$
f(x):= \sum_{i\geq 0} \ZZ(i)x^i =x.
$$
The first step of this article will be
to study a kind of constrained Galton-Watson process, in which the constraint is a "roof" at a given height that prevents the process from exploding.
\begin{defi}
Given $N$ a positive integer and a probability distribution $\ZZ$ on integers, 
the Galton-Watson process censored at level $N\geq 2$ with offspring $\ZZ$ is the process $(X_k^N)_{k\geq 0}$ defined by $X^N_0=N$ and, for $k\geq 0$,
$$
X_{k+1}^N=
\begin{cases}
&\min\left(N,\sum_{i=1}^{X_{k}^N}X_{k,i}\right)\text{ if }X_k^N>0,\\
&0 \text{ otherwise,}
\end{cases}
$$
where the $X_{k,i}$'s are i.i.d. copies of $\ZZ$.
\end{defi}

The process $(X_k^N)_{k\geq 0}$ is a finite state Markov process which is typically stuck for a long time on $N$, but eventually dies. 
Let $U_N$ be the \emph{survival time} of the Galton-Watson process :
$$
U_N=\min\{k\geq 0, X_k^N=0\}.
$$
Let us describe heuristically the asymptotic behavior of $U_N$. When the censored process dies, it happens very suddenly: in
the uncensored underlying process, the progenies of the $N$ particles pass away almost simultaneously in a few generations.
This latter event occurs with a probability close to $q^N$, and therefore the censored process is expected to survive a time close to a geometric random variable with parameter $q^N$.

We state this in the following theorem, which may be derived from results of  \cite{MR2761556}.

\begin{thm}\label{th1}
The following convergence holds in law:
\begin{equation}\label{Eq:ConvEnLoi}
U_N q^N\stackrel{N\to \infty}{\rightarrow} \mathcal E(1),
\end{equation}
where $\mathcal E(1)$ stands for the exponential distribution with parameter one. 
\noindent
The convergence also holds in mean:
\begin{equation}\label{Eq:LogEsperance}
\E(U_N) \sim (1/q)^N \hbox{ as }N\rightarrow \infty.
\end{equation}
\end{thm}
The convergence in law \eqref{Eq:ConvEnLoi} is a consequence of (\cite{MR2761556},Th.1 (3),(4)) and
the asymptotic \eqref{Eq:LogEsperance} follows from (\cite{MR2761556},Th.1 (2)). 
The results of \cite{MR2761556} are actually much accurate.
As we will only need these simple estimates, and for the sake of completeness, we provide in the two following sections an elementary and more concise 
proof of Theorem \ref{th1}. We also mention \cite{MR0267649}, in which the Galton-Watson process is censored by a function depending on time ; the author obtained a criterion for the degeneracy of the process.


\subsection{A connection with a branching-selection process}

The present authors considered Theorem~\ref{th1} when trying to find an asymptotic of the speed of a very particular case of a certain branching particle system with selection, 
studied by B\'erard and Gou\'er\'e (\cite{MR2669438}, Section~7, Theorem~5).
The \emph{branching-selection} process we will study is a generalization of their Bernoulli branching-selection process and can be described as follows. 
For a distribution $(\ZZ, \ZZ')$ on $\N^2$ 
such that the expectation of $\ZZ$ is strictly greater than $1$ and $\ZZ+\ZZ'\geq 1$ a.s., the particle system is the discrete-time particle system 
of $N$ particles moving on $\mathbb{Z}$ and starting from the origin such that, at each time unit,
\begin{enumerate}
\item Each of the $N$ particles is replaced by $\ZZ$ particles just on the right of that particle and by $\ZZ'$ particles on the same position, independently of each other ;
\item Among all these new particles, we keep only the $N$ rightmost particles.
\end{enumerate}

It is convenient to see the process of the locations of the $N$ particles as a finite point measure. We set $Y^N_0=N\delta_0$ and write
$$
Y^N_k = \sum_{\ell\geq 0} \delta_{\ell} Y^N_k(\ell),
$$
where $Y^N_k(\ell)$ is the number of particles at time $k$ at the position $\ell$ ; by construction we have at any time
$\sum_{\ell\geq 0}  Y^N_k(\ell)= N$. We write
$$
\max Y^N_k :=\max \set{\ell \geq 0 ; Y^N_k(\ell)>0}
$$
for the position of the rightmost particle at time $k$.
We have $\max Y_k^N\leq k$ by construction. 
Since $\E(\ZZ)>1$,  and since $Y_{k+1}^N\geq Y_k^N$ by the property of $\ZZ'$, it is likely for $N$ large that $\max Y_k^N$ is close to $k$. 
Remark also that in this process, all the particles at time $k$ are on $\{\max Y^N_k-1, \max Y^N_k\}$ after the selection.

B\'erard and Gou\'er\'e in fact studied the particle system defined as follows. At every time unit 
each particle is duplicated and moves one step forward with probability $\alpha\in(\frac12, 1)$ and stays put with probability $1-\alpha$.
With our notations, this corresponds to the case where $\ZZ$ is a Bernoulli variable $B(2,\alpha)$ 
with $\alpha>\frac12$, and $\ZZ'=2-\ZZ$ (we will denote their model as the \emph{Bernoulli branching-selection process}).
They proved that $(\max Y_k^N)/k$ converges almost surely to a constant $v_N(\alpha)$, and noticed that $v_N(\alpha)$ lies between $(1-\exp(-c_1(\alpha) N))$ 
and $(1-\exp(-c_2(\alpha) N))$ for some positive constants $c_1(\alpha)$ and $c_2(\alpha)$. 
The stress of their paper was a large class of distributions for the walk performed by the particles, 
and the case of the Bernoulli random walk with $\alpha >1/2$ was in fact just mentioned as a degenerate case. 
The existence of the asymptotic speed $v_N(\alpha)$ would apply as well in our case, but not their proof for the bounds on this speed.

The key point for our result is that the process of the number of particles that are at the rightmost possible position $(Y_k^N(k))_k$ 
has the same law as the censored Galton-Watson process $(X^N_k)_k$, when the offspring distribution is $\ZZ$.

\begin{thm}\label{th2} 
For the branching-selection process with the distribution $(\ZZ, \ZZ')$ on $\N^2$ such that $\E(\ZZ)>1$ and $\ZZ+\ZZ'\geq 1$ a.s., we have 
$$
\lim_{N\to\infty} \frac{1-v_N(\alpha)}{q^N}= 1,
$$
where 
$q$ is the extinction probability of a Galton-Watson process whose offspring is $\ZZ$.
\end{thm}
In particular, for the Bernoulli branching-selection model studied in \cite{MR2669438} with $\alpha>\frac12$,
$$
\lim_{N\to\infty} \frac{1-v_N(\alpha)}{(q_\alpha)^N}= 1,
$$
where 
$$
q_\alpha= \frac{1-2\alpha(1-\alpha)-\sqrt{1-4\alpha(1-\alpha)}}{2\alpha^2}
$$
is the extinction probability of a Galton-Watson process whose offspring is the Binomial distribution with parameters $(2,\alpha$).

Remark that the exact law of $\ZZ'$ has no influence, and we could have always taken $$\ZZ'=1_{\ZZ=0}.$$ 
Using $\ZZ'$ allows us essentially to have a true generalization of the Bernoulli branching-selection process.

The article is organized as follows. The proof of Theorem~\ref{th1} is decomposed over Sections~\ref{sectun} and \ref{secproof}. 
A key ingredient in Section \ref{secproof} is to compare our censored process at the last time it reaches the level $N$, and the classical Galton-Watson process starting at $N$ and conditioned to die.
The main contribution is the application to the branching-selection particle system and is given in Section~\ref{secdeux}. 

\section{Preliminaries}\label{sectun}

We will repeatedly use the following left-tail bound for a sum of copies of $\ZZ$ :
\begin{lem}\label{lemsimple}
Let $X_1,\dots, X_N$ be i.i.d. copies of law $\ZZ$. There exist $b,c>0$ such that, when $N$ is large enough,
$$
P\left(X_1+\dots + X_N\leq N(1+b)\right)\leq \exp(-cN).
$$
\end{lem}
To prove this, we choose $M$ such that $\mathbb{E}[\min\set{X_1,M}]>1$. We then can apply (\cite{MR0388499}, III Chap.4) to the bounded random variables $\min\set{X_i,M}$ and we get the desired bound.
This implies in particular that
\begin{equation}\label{Eq:majoX1}
\mathbb{P}(X_{k+1}^N<N\ |\ X_k^N=N)\leq \exp(-cN).
\end{equation}

Note that throughout the paper, $c$ stands for a generic positive constant, which might differ at each appearance.
We introduce the last time for which our process is equal to $N$. This variable will turn out to be equivalent to $U_N$ and more 
simple to approximate by a geometric variable.
\begin{defi}
Let 
$$
V_N=\max \{k\geq 0, X_k^N=N\},
$$
and let $q_N$ be the probability that $(X_k^N)_{k}$ does not ever hit $N$ after time zero:
$$
q_N=\Prob(X_k^N<N, \forall k>0)=\Prob(V_N=0).
$$
\end{defi}
We deduce from \eqref{Eq:majoX1} that
\begin{equation}\label{Eq:Majo_pN}
q_N\leq \Prob(X_1^N <N) \leq \exp(-c N).
\end{equation}

We will compare $q_N$ with the probability that the classical Galton-Watson process $(Z^N_k)$ starting from $N$ particles dies out. By independence of the progeny of the $N$ particles, we have
$\Prob((Z_k^N)\hbox{ dies out })=q^N$.
By conditioning over the first passage over some level $n$, we get
\begin{equation}
\Prob((Z_k)\hbox{ dies out }\mid \exists k\geq 0\hbox{ such that } Z_k \geq n)\leq q^n.
\end{equation}

\section{Proof of theorem \ref{th1}}\label{secproof}

The main tool of the proof is the following lemma, which says that the Galton-Watson process $(Z_k^N)$,
conditioned to die, roughly behaves like $X_k^N$ after its last passage in $N$.
Recall that $q$ stands for the probability that $Z$ dies while $q_N$ is the probability that $X^N$ does not ever hit $N$.

\begin{lem}\label{lempN}
There exists $c>0$ such that for $N$ large enough
$$
 \left|\frac{q_N}{q^N}-1\right|\leq \exp(-cN).
$$
\end{lem}

\begin{proof}
By Lemma \ref{lemsimple}, we can take $b>0$ and $c>0$ such that
$$
 \Prob(Z_1^N<(1+b)N)\leq \exp(-cN).
$$
The point is that if $(Z_k^N)_k$ dies out, then with high probability it does not hit $N$ after time $k=0$.
Let $R=\inf\set{k\geq 1,Z_k^N\geq N}$ (with the convention that $R=+\infty$ is $Z^n$ does not hit $N$ after time zero),
\begin{eqnarray*}
\lefteqn{\Prob(R<\infty ;\ Z^N \Zdies)}\\
&\leq& \Prob( Z_1^N< (1+b)N;R<\infty ;\ Z^N \Zdies) +\Prob(Z_1^N\geq (1+b)N;\ Z^N \Zdies)\\
&\leq&  \Prob( Z_1^N< (1+b)N;\ R<\infty ;\ Z^N \Zdies) + q^{(1+b)N}\\
&\leq&  \sum_{\ell\geq N} \Prob(Z_R^N=\ell;\ R<+\infty;\ Z_1^N< (1+b)N ;\ Z^N \Zdies) + q^{(1+b)N}\\
&\leq&  \sum_{\ell\geq N} \Prob( Z^N \Zdiest|\ Z_R^N=\ell; R<+\infty; Z_1^N< (1+b)N) \Prob( Z_R=\ell; R<+\infty; Z_1^N<(1+b)N)\\
& & \qquad+ q^{(1+b)N}\\
&\leq&  \sum_{\ell\geq N} \Prob( Z^N \Zdiest|\ Z_R^N=\ell;\ R<+\infty) \Prob(Z_1^N< (1+b)N)+ q^{(1+b)N}\\
&\leq&  \sum_{\ell\geq N} q^\ell e^{-cN}+ q^{(1+b)N} \leq  \frac1{1-q}\exp(-cN)q^N + q^{(1+b)N}.
\end{eqnarray*}
Finally, for $N$ large enough
\begin{eqnarray}\label{Eq:ZNVN}
\Prob(Z^N \Zdies)-\Prob(V_N=0) & =&\Prob(Z^N \Zdies, R< \infty )\nonumber\\
&\leq &q^N\exp(-c'N),
\end{eqnarray}
which finishes the proof.
\end{proof}

By construction $U_N\geq V_N$, we first prove that $U_N$ is close to $V_N$, in the following sense.
\begin{lem}\label{lemUV}
The sequence $(\frac{U_N}{1+ V_N})_{N\geq 1}$ converges to $1$ in probability.
\end{lem}
\begin{proof}
By definition we have $U_N\geq V_N +1$. For $\eps>0$,
\begin{eqnarray*}
 \Prob\left(\frac{U_N}{1+V_N}-1 \leq -\eps\right)\leq\Prob(U_N-(V_N+1)>\sqrt{\eps} N)+\Prob\left(U_N<\frac1{\sqrt{\eps}} N\right).
\end{eqnarray*}

We study the first term using a standard technique in branching processes theory:
$$
\Prob((Z_k^1)\text{ dies };Z_K >0)=q-f_K(0),
$$
where $f_K$ is both the $K$-th iterate of $f$ and the generating function of $Z_K$. Our assumptions imply
that $|q-f_K(0)|\leq \exp(-cK)$ for some $c>0$ (\cite{MR2047480},Th.1,chap.I.11).

By the Markov property, the process $(X^N_{k+\ell})_{\ell\geq 0}$ conditioned to $X^N_k=N$ has the same law as $(X^N_\ell)_{\ell\geq 0}$, it follows that
$$
\Prob(U_N-(V_N+1) > K) = \Prob(U_N >K +1 \mid V_N=0).
$$
Recall \eqref{Eq:ZNVN} :
\begin{eqnarray*}
\Prob(V_N=0) \geq \Prob(Z^N \Zdies) - 2q^N\exp(-c'N) \geq q^N/2
\end{eqnarray*}
for $N$ large enough.
We then write
\begin{eqnarray}
\Prob(U_N >K +1\mid V_N=0)&=&\Prob(Z_{K+1}^N>0\mid Z_k^N <N \text{ for each }k\geq 1 )\notag\\
&\leq &\Prob(Z_{K+1}^N>0\mid (Z_k^N) \hbox{ dies out })\times\frac{\Prob\left((Z_k^N) \hbox{ dies out }\right)}{\Prob(Z_k^N <N \text{ for each }k\geq 1)}\notag\\
&\leq &N\Prob(Z_{K+1}^1>0\mid (Z_k^N) \hbox{ dies out })\times \frac{q^N}{q^N/2}\notag\\
&\leq &2N\frac{\exp(-cK)}{q},\label{Eq:MajoUN-VN}
\end{eqnarray}
where we finally used (\cite{MR2047480}, Sec. I.11). We obtain 
\begin{equation*}
 \Prob(U_N-V_N-1>\sqrt{\eps} N)\leq CN\exp(-c\sqrt{\eps} N)
\end{equation*}
for some $C$ and $c>0$ and for any $N\geq 1$ and $\eps>0$.

\noindent
The second term is handled thanks to \eqref{Eq:Majo_pN} as follows. For $N> 1/\eps$,
\begin{eqnarray*}
\Prob\left(U_N<\lf \frac1{\sqrt{\eps} } N\rf\right)&\leq& \Prob\left(V_N<\lf \frac1{\sqrt{\eps}} N\rf \right)\\
&\leq& \Prob\left(\exists v\leq \lf \frac1{\sqrt{\eps}} N\rf , V_N=v\right)\\
&\leq& \Prob\left(\exists v\leq \lf \frac1{\sqrt{\eps}} N\rf , X_v^N=N \text{ and } X_{v+1}^N<N\right)\\
&\leq&\frac1{\sqrt{\eps}} N \exp(-cN)
\end{eqnarray*}
for a certain $c>0$. Letting $N$ go to infinity finishes the proof.
\end{proof}

We introduce another process and an associated time whose properly renormalized law will converge and which is equivalent to $V_N$.
Let $A_k$ be the $k$-th passage in $N$ of the censored process: $A_0=0$ and for $k\geq0$, $A_{k+1}=\inf\{\ell >A_k, X_\ell^N=N\}$,
with the usual convention that $\inf(\emptyset)=+\infty$. Let $T$ be the survival time of $A$ :
$$
T=\sup\{k\geq 0, A_k<\infty\}.
$$
It is clear by contruction that $T$ has the law of $\mathcal{G}(q_N)-1$, where $\mathcal{G}(r)$ is a geometric of parameter $r$.
By the Markov property of our branching process, the variables $A_{k+1}-A_k$ are independent and identically distributed. With high probability, $A_{k+1}=A_k+1$, we can be more precise:
\begin{lem}\label{lemykn}
There exists $c>0$ such that for all $N$
$$
\Prob(A_{k+1}>A_k+1)\leq\exp(-c N).
$$
There exists $c,C>0$ such that for all $K\geq 1$, for all $N$ large enough,
$$
 \Prob(A_k+K< A_{k+1}<\infty)\leq C\exp\left(-c\frac{K}{\log N}-cN\right).
$$
\end{lem}

\begin{proof}
The first assertion is just \eqref{Eq:majoX1}.
The second one is a consequence of the following inequality: for any $0< i< N$
\begin{equation}\label{eqyk2}
\Prob(0<Z_k^i<N \text{ for all }k\leq K)\leq C\exp\left(-c\frac{K}{\log N}\right).
\end{equation}
To prove so, take $1<\mu_\star <\mu \leq +\infty$. The sequence
$Z_k/\mu_\star^k$ tends to infinity with positive probability (see \cite{MR2047480},Th.3 chap.I.10). 
Hence we have positive constants $\delta,\beta$ such that, for any integers $k,n$,
\begin{equation}\label{Eq:Depassement}
\mathbb{P}(Z_k^n > \delta \mu_\star^k)\geq \mathbb{P}(Z_k^1 > \delta \mu_\star^k) \geq \beta.
\end{equation}
Now, set $m =\lceil \log(N/\delta)/\log(\mu_\star)\rceil $. We assume that $N$ is large enough, such that $m\geq 1$.
We first assume that $K> m$.
\begin{multline*}
\set{0<Z_k<N \text{ for all }k\leq K}
\subset 
\set{Z_m<N} \cap \set{Z_m>0 \text{ and }Z_{2m} <N} \\
\cap\set{Z_{2m}>0 \text{ and }Z_{3m} <N} \cap \dots \cap
\set{Z_{\lf K/m\rf m-m}>0 \text{ and }Z_{\lf K/m\rf m} <N}.
\end{multline*}
Thanks to \eqref{Eq:Depassement}, the probability of the right-hand side is smaller than $(1-\beta)^{\lf K/m\rf}$.
If $K\leq m$ then $K/\log N\leq 2/\log\mu^\star$ for $N$ large enough. By choosing $C$ large enough, the right-hand side in 
\eqref{eqyk2} is greater than one, and the claimed inequality also holds.

To finish the proof of the lemma, it suffices to remark that for $K\geq 1$,
\begin{align*}
 \Prob(A_k+K< A_{k+1}<\infty)&\leq \Prob(A_k+1< A_{k+1}<\infty)\Prob(0<Z_\ell^{N} < N \ \forall \ell\leq K)\\
&\leq C\exp(-cN) \exp\left(-c'\frac{K}{\log N}\right).
\end{align*}
\end{proof}
\begin{lem}\label{lemlim}
The sequence $\displaystyle{\left(\frac{1+V_N}{1+T}\right)_{N\geq 1}}$ converges to one in probability.
\end{lem}
\begin{proof}

By construction we have $T\leq V_N$. In order to bound the difference, we first notice that, $T$ being a $\mathcal{G}(q_N)-1$, we have for large enough $N$ the bound
\begin{equation}\label{eqTqN}
\Prob(T < 1/\sqrt{q_N})\leq (1/\sqrt{q_N}+1)\Prob(T=0)\leq (1/\sqrt{q_N}+1)q_N \leq 2\sqrt{q_N}\leq 3q^{N/2},
\end{equation}
where we finally used Lemma \ref{lempN}. We now write
$$
V_N\stackrel{\text{(law)}}{=} \sum_{k=0}^T A_k \stackrel{\text{(law)}}{=} T+ \sum_{k=0}^T \widetilde A_k,
$$
where the $(\widetilde A_k)$ is a sequence of i.i.d. random variables with the same law as $A_1-1$ conditioned by $A_1<\infty$, and the sequence
is independent from $T$. 

Now take $\eps>0$. By Lemma~\ref{lemykn}, for $N$ large enough we have $\E(\widetilde A_1)\leq \eps/2$.
\begin{align*}
\Prob\left(V_N+1-(T+1)\geq\right.  \eps &\left. (T+1)\right) 
= \Prob\left(\sum_{k=0}^{T}\widetilde A_k\geq \eps (T+1)\right)\\
&\leq  \Prob\left(\sum_{k=0}^{T}\widetilde A_k-\E(\widetilde A_k)\geq (T+1)\eps /2\right)\\
&\leq \Prob\left(\sum_{k=0}^{T}\widetilde A_k-\E(\widetilde A_k)\geq \eps (T+1)/2 ; T\geq 1/\sqrt{q_N}\right)+ \Prob\left(T< 1/\sqrt{q_N}\right)\\
&\leq \sum_{t\geq 1/\sqrt{q_N}} \Prob\left(\sum_{k=0}^{t}\widetilde A_k-\E(\widetilde A_k)\geq \eps (t+1)/2\right)+ \Prob\left(T< 1/\sqrt{q_N}\right)\\
&\leq \sum_{t\geq 1/2\sqrt{q^N}} \Prob\left(\sum_{k=0}^{t}\widetilde A_k-\E(\widetilde A_k)\geq \eps (t+1)/2\right)+ \Prob\left(T< 1/\sqrt{q_N}\right),  
\end{align*}
where we finally used Lemma \ref{lempN}.
The second term on the right-hand side goes to zero thanks to \eqref{eqTqN}. Let us now handle the sum. Using Lemma \ref{lemykn} we have 
\begin{equation*}
\Prob\left(\widetilde A_k \geq K\right)=\Prob(K\leq A_1<\infty)\Prob(A_1 <\infty)^{-1}\leq 2\exp(-cK/\log N)
\end{equation*}
for a certain $c>0$.
From such a tail, we deduce (see \cite{MR0388499}, sec.III.4) that
there exists $c>0$ such that for all $t$, $\eps$ and $N$,
$$
\Prob\left(\sum_{k=0}^{t}\widetilde A_k -\E(\widetilde A_k)\geq \eps t/2\right)\leq \exp\left(-c\eps \frac{t}{\log N}\right),
$$
and the sum goes to zero.
\end{proof}

\begin{proof}[End of proof of Theorem \ref{th1}]
We now write
$$
U_Nq^N =\frac{U_N}{V_N+1} \frac{q^N}{q_N} q_N (T+1)\frac{V_N+1}{T+1},
$$
where, when $N$ goes to infinity,

\begin{tabular}{ r l l }
$U_N/( V_N+1)$  & $\to 1$  & (in prob.) by Lemma \ref{lemUV},\\
$q^N /q_N$  & $\to 1$  & by Lemma \ref{lempN},\\
$q_N/ (T+1)$            & $\to \mathcal E(1)$ & (in law) since $T$ is geometric with parameter $q_N$,\\
$(V_N+1)/(T+1)$    & $\to 1$ & (in prob.) by Lemma \ref{lemlim}.
\end{tabular}

\noindent
This proves the convergence in distribution. For the convergence in mean, recall that one may write
$$
\mathbb{E}[V_N]=\mathbb{E}[ T+ \sum_{k=0}^T \widetilde A_k]=\mathbb{E}[T]\left(1+ \mathbb{E}[\widetilde A_1]\right)=q^N(1+\mathrm{o(1)}),
$$
thanks to Lemma \ref{lempN} and \ref{lemykn}. Since $\mathbb{E}[U_N]=\mathbb{E}[V_N]+\mathbb{E}[U_N-V_N]$, we finally obtain with \eqref{Eq:MajoUN-VN}
\begin{equation}\label{Eq:EsperanceUnVn}
\mathbb{E}[U_N]\sim \mathbb{E}[V_N]\sim (1/q)^N,
\end{equation}
which will be useful in the next section.
\end{proof}

\section{Application: Branching-selection particle system}\label{secdeux}

We now describe more precisely the connection with the branching-selection process defined in the introduction, in order
to exploit the results of the previous section.
Recall that $Y^N_k(\ell)$ is the number of particles at time $k$ at the position $\ell$ and that
$v_N:=\lim_{k\to\infty} (\max Y_k^N)/k$ (a.s.).

The connection with the censored Galton-Watson Process in is the following Lemma, whose proof is immediate by construction of the process.
\begin{lem}
In the branching-selection process with a pair $(\ZZ, \ZZ')$ of laws on $\N$, 
the number of particles that are at the rightmost possible position has the law of a censored Galton-Watson process, 
when the offspring distribution is $\ZZ$ :
$$
(Y^N_k(k))_{k\geq 0} \stackrel{\text{law}}{=}(X^N_k)_{k\geq 0}.
$$
\end{lem}
In particular, if we define $V^1$ as the last time at which the $N$ particles are at the rightmost possible location:
$$
V^1=\max\set{k; Y^N_k(k)=N},
$$
then $V^1\stackrel{\text{law}}{=}V_N$, where $V_N$, defined in the previous section, is the last time at which the censored Galton-Watson hits $N$.

We now are ready to estimate $v_N$.

\begin{lem} For all integer $N\geq 1$,\label{lemvnv}
\begin{equation*}
  v_N\geq 1-\frac1{E(V_N)+1},
\end{equation*}
where the progeny distribution defining $V_N$ is $\ZZ$.
\end{lem}

\begin{proof}
The proof is inspired by that of Proposition 4 in \cite{MR2669438}. 
The main idea is to design a \emph{dominated} process $\tilde{Y}_k$ moving slower than the \emph{true} process $Y_k$.

Let us skip the exponent $N$ in order to lighten the notations, and consider an i.i.d. family $(\ZZ_{\ell,k,i}, \ZZ'_{\ell, k, i})_{\ell\geq 0, k\geq 0, i\in [1,N]}$ 
of integer variables with the same law as $(\ZZ, \ZZ')$. 

We sample the particle system $(Y_k)$ by means of these variables:
the population $Y_k$ at time $k$ being given, we let
\begin{equation*}
T_{k+1}=\sum_{\ell=0}^\infty\sum_{i=1}^{Y_k(\ell)}\left(\ZZ_{\ell, k, i}\delta_{\ell+1}+\ZZ'_{\ell,k,i}\delta_\ell\right)
\end{equation*}
be the locations of the particles after branching. The population $Y_{k+1}$ is then obtained by keeping only the $N$ rightmost particles among $T_{k+1}$.


\begin{figure}
\label{Fig:Particules}
\begin{center}
\includegraphics[width=14cm]{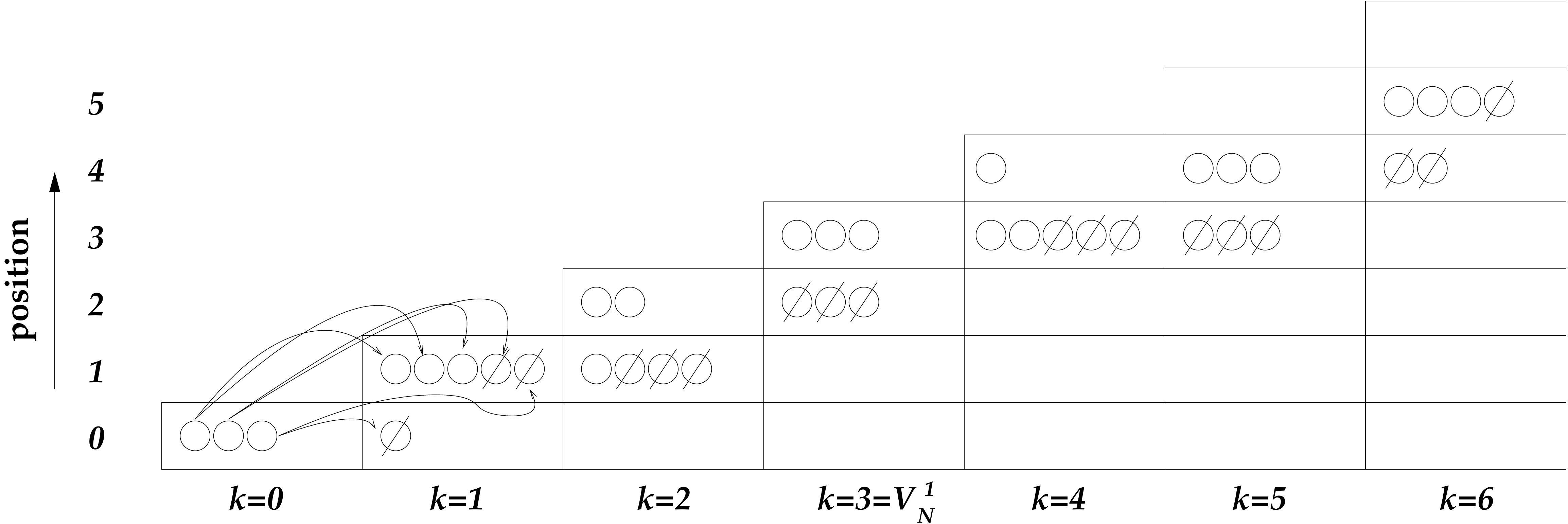}
\caption{A realization of the \emph{true} process $Y$ with $N=3$ particles, for the particular case of the Bernoulli branching-selection process. 
At each step are represented the $6$ particles (after replication) and the $3$ leftmost particles are striked.
Here, we see that $V_1=3$, since from time $k=5$ there can be no more particles at
the rightmost possible position, and because at time $4$ only one particle
is on the rightmost possible position.}
\end{center}
\end{figure}

We modify $Y$ in the following way.
At time $V^1+1$, let continue the process as if all the particles were at position $V^1$ at time $V^1+1$. 
More precisely, set $(Y^{(1)}_k)=(Y_k)$ and introduce a process $\left(Y^{(2)}_k\right)_{k\geq V^1}$ starting from
$$
Y_{V^1+1}^{(2)}=N\delta_{V^1}.
$$
Then, for $k\geq V^1+1$, $Y_{k+1}^{(2)}$ is sampled from $Y_{k}^{(2)}$ exactly as $Y_{k+1}$ is sampled from $Y_{k}^{N}$ (with the same $(\ZZ_{\ell,k,i}, \ZZ'_{\ell, k, i})$'s). 
The main point is that, at time $V^1 +1$, the $N$ particles of $Y^{(1)}_{V^1 +1}$ are at positions $V^1$ or $V^1+1$ while the $N$ particles of $Y^{(2)}_{V^1 +1}$ are at position $V^1$. 
Hence, the point measure $Y^{(1)}_{V^1 +1}$ dominates 
$Y^{(2)}_{V^1 +1}$, and this domination will continue throughout the process, since the same $(\ZZ,\ZZ') $'s are used to generate $Y^{(1)}$ and $Y^{(2)}$.

Now, let $V^2$ be such that $V^1+1+V^2$ is the last time $k$ at which the $N$ particles are at the rightmost possible location (which is position $k-1$) for $Y^{(2)}$:
$$
V^2=\max\set{k; Y^{(2)}_k(k-1)=N}-(V^1+1).
$$
The random time $V^2$ has the same law as $V^1$ and, as a function of $\set{\ZZ_{\ell,k,i},\ZZ'_{\ell, k, i}; \ell >V^1}$, is independent of $V^1$.
We define recursively similar processes $Y^{(3)},Y^{(4)},\dots$ and random variables $V^3,V^4,\dots$. 
Setting also
$$
\Gamma^0=0 ,\quad \Gamma^{i+1}=\Gamma^i +V^i+1,
$$
the $\Gamma^i$'s are renewals. 

Let us introduce a last process $(\tilde{Y}_k)_{k\geq 0}$ as follows:
if $k\in [\Gamma^{i-1},\Gamma^i)$, then
$$
\tilde{Y}_k=Y^{(i)}_k
$$
\begin{figure}
\label{Fig:ParticulesModifiees}
\begin{center}
\includegraphics[width=14cm]{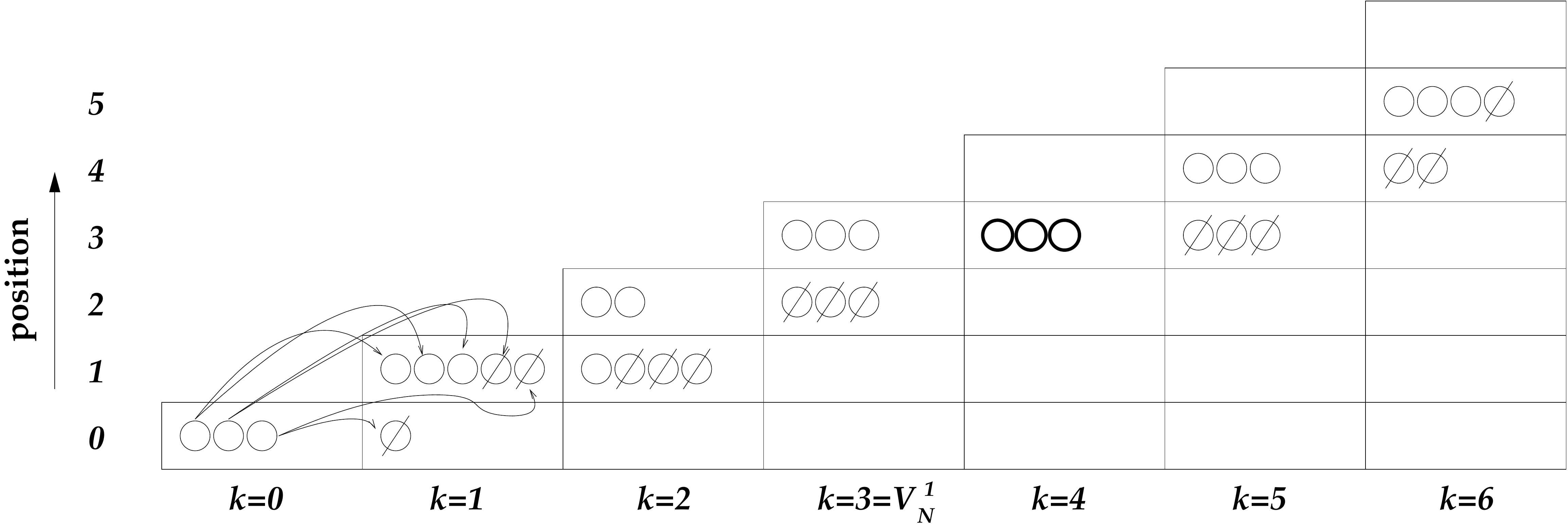}
\caption{A realization of the associated \emph{dominated} process: at time $V^1_N+1=4$, we let the process restart with {\bf all} particles at position $3$.}
\end{center}
\end{figure}

\noindent
One can see in Fig.~1 the first steps of a realization of the process $(Y_k)$, and in Fig.~2 the 
corresponding modificated process $(\tilde{Y}_k)$.
By construction, $Y_k$ dominates $\tilde{Y}_k$ for each $k$, and in particular $\max Y_{k}\geq \max\widetilde Y_{k}$.
Let us also note that at each renewal $\Gamma^i$ the quantity $\max\widetilde Y_{k}$ is decreased by one. 
Set $I_k$ be the renewal process associated to the renewals $\Gamma^i$, that is $I_k$ is the only integer $i$ such that 
$\Gamma^i \leq k< \Gamma^{i+1}.$ On the one hand we have
$$\Gamma^{i+1}-\Gamma^i\stackrel{\text{law}}{=} V^1+1,
$$
so by applying the renewal theorem (see for instance \cite{MR1609153} Chap.3.4) to the renewals $\Gamma^1,\Gamma^2,\dots$, we get
\begin{equation*}
\lim_{k\rightarrow\infty} \frac1k\E(I_k)\rightarrow\frac{1}{\mathbb{E}[V^1]+1}.
\end{equation*}
On the second hand we have
$$
\max\widetilde Y_{k} = k - I_k.
$$
Hence we get
$$
\frac{1}{k}\mathbb{E}[\max Y_{k}]\geq \frac{1}{k}\mathbb{E}[\max\widetilde Y_{k}]\stackrel{k\to\infty}\rightarrow 1-\frac{1}{\mathbb{E}[V^1]+1}, 
$$

To conclude, recall that thanks to the connection with the censored Galton-Watson process $\mathbb{E}[V^1]=\mathbb{E}[V_N]$.
\end{proof}

\begin{lem}\label{lemvnu} For all integer $N\geq 1$,
\begin{equation*}
 v_N\leq 1-\frac1{E(U_N)}
\end{equation*}
\end{lem}
\begin{proof}
The proof is an adaptation of the previous one, replacing $V^1$ by $U^1$, which is the first time at which there is no more particle at the rightmost possible position:
$$
U^1=\min\set{k; Y^N_k(k)=0},
$$ 
(On the example drawn in Fig.~1, one has $U^1=5$.) Then, $U^1$ has the same law as the $U_N$ defined in Section \ref{sectun}. In a similar manner to the previous proof, we let the process restart at time $U^1$ as if all the particles were at position $U^1-1$. The end of the proof is similar.
\end{proof}

We now combine these two estimates of $v_N$ with the results of the previous section in order to prove Theorem \ref{th2} :
$$
\frac1{E(U_N)} \leq  1- v_N \leq \frac1{E(V_N)+1}
$$
and both sides, thanks to \eqref{Eq:EsperanceUnVn}, are equivalent to $q^N$.


\vspace{5mm}

\noindent{\bf Aknowledgement.} We warmly thank J.-B. Gou\'er\'e for many helpful comments on this work.
We are also very thankful to the two referees and the associate editor for recommending numerous improvements.
\bibliography{Galton}

\end{document}